\documentclass[11pt]{article}

\usepackage[english]{babel}
 \usepackage[T1]{fontenc}
\usepackage{fancyhdr}
\usepackage{indentfirst}
\usepackage{graphicx}
\usepackage{newlfont}
\usepackage{cite}
\usepackage{lipsum}
\usepackage{color}

\newcommand\blfootnote[1]{%
  \begingroup
  \renewcommand\thefootnote{}\footnote{#1}%
  \addtocounter{footnote}{-1}%
  \endgroup
}

\usepackage{float}
\usepackage{wrapfig}
\usepackage{array}

\usepackage{amsfonts}
\usepackage{amssymb}
\usepackage{amsmath}
\usepackage{latexsym}
\usepackage{amsthm}
\usepackage[margin=1.2in]{geometry}

\DeclareMathOperator{\diam}{diam}
\theoremstyle{plain}                   
\newtheorem{theorem}{Theorem}

\newtheorem{ex}[theorem]{Example}
\theoremstyle{remark}
\newtheorem*{rem}{Remark}
\newcommand{\R}{\mathbb{R}}

\begin{document}
\title {Geometric properties of extremal sets for the dimension comparison in Carnot groups of step 2}
\author{Laura Venieri}

\maketitle 

\begin{abstract}
In Carnot groups of step 2 we consider sets having maximal or minimal possible homogeneous Hausdorff dimension compared to their Euclidean one: in the first case we prove that they must be in a sense vertical, that is a large part of these sets can lie off horizontal planes through their points, and in the latter case close to horizontal. The examples that had been used to prove the sharpness of the sub-Riemannian versus Euclidean dimension comparison intuitively satisfied these geometric properties, which we here quantify and prove to be necessary. We also show the sharpness of our statements with some examples. This paper generalizes the analogue results proved by Mattila and the author in the first Heisenberg group. 
\end{abstract}

\blfootnote{\textit{Key words.} Hausdorff measure, Carnot groups, Hausdorff dimension, dimension comparison.\\
\textit{Mathematics Subject Classification}. 28A75.\\
The author was supported by the Academy of Finland through the Finnish Center of Excellence in Analysis and Dynamics research and through the project \emph{Quantitative rectifiability in Euclidean and non-Euclidean spaces}, grant no. $314172$.}

\section{Introduction}

Carnot groups are simply connected nilpotent Lie groups whose Lie algebra admits certain gradings, called stratifications. The first layer of the stratification is called horizontal and plays a special role since it generates all the Lie algebra by commutations. Using the stratification, one can define a metric, called Carnot-Carath\'eodory metric, which is of sub-Riemannian type, and a natural family of dilations. The Carnot-Carath\'eodory metric is bi-Lipschitz equivalent to any metric that is left-invariant with respect to the group operation and one-homogeneous with respect to the dilations. We call such a metric homogeneous.

Carnot groups appear in several areas of mathematics, such as harmonic analysis, where they play a key role in the study of hypoelliptic differential operators (see e.g. \cite{RS} and \cite{S}).
Carnot groups are metric tangents of sub-Riemannian manifolds, which are manifolds endowed with a linear subbundle of the tangent bundle and an inner product on it, and have for example applications in physics (see \cite{Mo}) and in modelling the visual cortex (see \cite{CS}). 

Any Carnot group $\mathbb{G}$ is homeomorphic to $\R^n$ for a certain $n$, which means that topologically they are the same. The Euclidean metric, however, is not bi-Lipschitz equivalent to any homogeneous metric, thus, for instance, the Hausdorff dimension of a subset of $\mathbb{G}$ in the homogeneous and in the Euclidean metric can differ. We denote by $\dim_\mathbb{G}$ and $\dim_E$ the Hausdorff dimension in the homogenous and Euclidean sense respectively. In \cite{BTW} (Theorem 2.4) Balogh, Tyson and Warhurst solved the dimension comparison problem, determining explicit functions $\beta_{\pm}$ such that for every $A \subset \mathbb{G}$,
\begin{displaymath}
\beta_-(\dim_E A) \le \dim_\mathbb{G}A \le \beta_+(\dim_E A).
\end{displaymath}
They also provide examples to show the sharpness of these inequalities, developing a theory of fractal geometry in Carnot groups. Intuitively sets with $\dim_\mathbb{G}A =\beta_+(\dim_E A)$ are as vertical as possible, that is they lie in the direction of the higher strata of the Lie algebra. Sets with $\dim_\mathbb{G}A =\beta_-(\dim_E A)$, instead, lie in the direction of the lower strata, that is they are as horizontal as possible. The tecnique used for constructing examples in the latter case worked only for a dense set of dimensions and only for those the examples have positive and finite Hausdorff measures at the respective dimensions. Later, Rajala and Vilppolainen \cite{RV} constructed examples of such sets having positive and finite measures for any dimension, replacing self-similar constructions with so-called controlled Moran constructions.
 
This paper is a sequel of \cite{MV}, where Mattila and the author gave a characterization of extremal sets for the dimension comparison in the first Heisenberg group $\mathbb{H}^1$. We now prove analogue results in Carnot groups of step $2$, which have a stratification consisting of two layers. In this case the functions $\beta_{\pm}$ have the form 
\begin{equation}\label{beta}
\beta_-(s)=\max\{s,2s-m_1\}, \quad  \beta_+(s)=\min\{2s,s+m_2\},
\end{equation}
where $m_1$ is the dimension of the first layer of the group (called horizontal) and $m_2$ of the second (see Section \ref{step2prel} for more details). The Heisenberg group $\mathbb{H}^1$, identified with $\R^3$, is the simplest example of a Carnot group of step $2$, where $m_1=2$ and $m_2=1$.

Theorem \ref{thm1} states that when $0<s\le m_1$, which implies $\beta_-(s)=s$, sets having positive and finite Euclidean and homogeneous $s$-Hausdorff measures lie close to horizontal $m_1$-planes in some arbitrarily small neighbourhoods of almost every point. Horizontal $m_1$-planes, which are obtained by left-translating the horizontal layer, are defined in \eqref{horplane}. We show in Examples \ref{ex1s1}, \ref{ex2s1} that this result is sharp at least when $s=1$ in the sense that we cannot take all small neighbourhoods. Examples \ref{ex1} and \ref{ex2} show the sharpness in the case of Heisenberg groups $\mathbb{H}^m$, identified with $\R^{2m+1}$, when $s$ is half the dimension of the first layer $m_1=2m$. Example \ref{ex3}, which relies on the construction in \cite{RV}, shows that when $s \ge m_1$ the same result does not hold since a large portion of the set can lie off horizontal planes through its points. In Theorems \ref{sgreat} and \ref{ssmall1} we consider the case of sets with positive and finite $s$-dimensional Euclidean Hausdorff measure and $\beta_+(s)$-dimensional homogeneous Hausdorff measure: they cannot lie too close to horizontal planes in small neighbourhoods of their typical points. Even if the proofs of the theorems are an adaptation of those in \cite{MV}, the results are of independent interest and we explain in Section \ref{stepk} that the result corresponding to Theorem \ref{ssmall1} holds also in general Carnot groups.

\subsection{Acknowledgements}
I would like to thank Pertti Mattila for many fruitful discussions about the topic. 

\section{Preliminaries}\label{prel}
\subsection{Carnot groups of step $2$}\label{step2prel}
A Carnot group of step $2$ is a connected, simply connected Lie group whose Lie algebra $\mathfrak{g}$ admits a stratification of step $2$, that is there exist linear subspaces $V_1, V_2$ such that
\begin{displaymath}
\mathfrak{g}=V_1 \oplus V_2, \quad [V_1,V_1]=V_2, \ [V_1,V_2]=\{0\}, 
\end{displaymath}
where $[V_1,V_i]$ is spanned by the commutators $[X,Y]$ with $X \in V_1$, $Y \in V_i$. Let $m_1$, $m_2$ be the dimensions of $V_1$, $V_2$ respectively.

Using exponential coordinates, we can identify $\mathbb{G}$ with $\R^n=\R^{m_1}\times\R^{m_2}$ and denote points by $p=[p^1,p^2]=(x_1, \dots, x_{m_1}, x_{m_1+1}, \dots, x_n)$, where $p^i \in \R^{m_i}$, $i=1,2$, and $x_j \in \R$, $j=1, \dots,n$. The group operation has the form
\begin{equation*}
[p^1,p^2] \cdot [q^1,q^2]= [p^1+q^1, p^2+q^2+P(p,q)],
\end{equation*}
where $P=(P_{m_1+1}, \dots, P_n)$ and each $P_j$ is a polynomial of the form 
\begin{equation}\label{Pj}
P_j(p,q)=P_j(p^1,q^1)=\sum_{1\le l<i\le m_1} b^j_{l,i} (x_l y_i - x_i y_l),
\end{equation}
for some $b^j_{l,i} \in \R$ (see Lemma 1.7.2 in \cite{Mon}).
There is a natural family of dilations on $\mathbb{G}$ given by
\begin{displaymath}
\delta_\lambda(p)=[\lambda p^1, \lambda^2 p^2], \quad \lambda>0,
\end{displaymath}
and each polynomial $P_j$ is homogeneous of degree $2$ with respect to $\delta_\lambda$, e.g $P_j(\delta_\lambda(p), \delta_\lambda(q))=\lambda^2 P_j(p,q)$.
The simplest examples of Carnot groups of step $2$ are the Heisenberg groups $\mathbb{H}^m$, $m\ge 1$, identified with $\R^{2m+1}=\R^n$, for which $m_1=2m$ and $m_2=1$. Denoting the points by $p=(x,y,t)$ with $x,y \in \R^m$, $t \in \R$, the group operation takes the form
\begin{equation}\label{operHn}
(x,y,t) \cdot (x',y',t')=(x+x',y+y', t+t' +2(\langle y,x' \rangle-\langle x,y' \rangle)),
\end{equation}
where $\langle \cdot, \cdot \rangle$ denotes the scalar product.

On Carnot groups one can define the Carnot-Carath\'eodory metric using horizontal curves, which are tangent at almost every point to the vector fields generating $V_1$ (see e.g. Section 2.2 in \cite{BTW}). Here we will work instead with the metric
\begin{equation}\label{dinfty}
d_\infty(p,q)=\max \{|p^1-q^1|_{\R^{m_1}}, c |p^2-q^2+P(p^1,q^1)|_{\R^{m_2}}^{1/2} \},  
\end{equation}
where $c \in (0,1)$ is a suitable constant depending only on the group structure (see Theorem 5.1 in \cite{FSSC}). Here and in the following we will use $| \cdot |$ to denote the Euclidean norm and drop the subscripts $\R^{m_i}$. Like the Carnot-Carath\'eodory metric, $d_\infty$ is left-invariant with respect to the group operation and one-homogeneous with respect to the dilations, i.e.
\begin{displaymath}
d_\infty(p \cdot q, p \cdot q')=d_\infty(q,q'), \quad d_\infty(\delta_\lambda(p), \delta_\lambda(q))=\lambda d_\infty(p,q)
\end{displaymath}
for every $p,q,q' \in \mathbb{G}$ and $\lambda>0$. We call such metrics homogeneous and by Proposition 5.1.4 in \cite{BLU} they are all bi-Lipschitz equivalent. We denote by $B_E(p,r)$ and by $B_\infty(p,r)$ the closed ball of center $p$ and radius $r$ in the Euclidean metric $d_E$ and in the homogeneous metric $d_\infty$ respectively.
Proposition 5.15.1 in \cite{BLU} states that for every $0<R<\infty$ there exists $c_R>0$ such that for every $p,q \in B_E(0,R)$ we have
\begin{equation}\label{compd}
\frac{1}{c_R} d_E(p,q) \le d_\infty(p,q) \le c_R d_E(p,q)^{1/2}.
\end{equation}

The horizontal $m_1$-plane $V(p)$ passing through a point $p \in \mathbb{G}$ is the left translation by $p$ of $\{[q^1,0]: q^1 \in \mathbb{R}^{m_1}\}$. It is the set of points $q=[q^1,q^2] \in \mathbb{G}$ such that
\begin{equation}\label{horplane}
q^2=p^2+P(p^1,q^1).
\end{equation}
Note that if $q \in V(p)$ then 
\begin{equation}\label{dVp}
d_\infty(p,q)= |p^1-q^1| \le d_E(p,q)
\end{equation} 
We will denote by $V(p)(\delta)$ the $\delta$-neighbourhood of $V(p)$ in the Euclidean metric, that is the set of points $q \in \mathbb{G}$ whose Euclidean distance from $V(p)$ is $\le \delta$.
Note that if $p\in B_E(0,R)$ and $0<r<1$, then
\begin{equation}\label{Vpneighb}
V(p) \left(\frac{r^2}{c_R^2}\right) \cap B_E(p,r)  \subset B_\infty(p,r) \subset V(p)\left(\frac{r^2}{c^2}\right) \cap B_E(p, c_R r).
\end{equation}
To see that the right-hand side inclusion holds, let $q=[q^1,q^2] \in B_\infty(p,r)$. Then $d_E(p,q)\le c_R d_\infty(p,q) \le c_R r$ and 
\begin{align}\label{distVp}
d_E(q, V(p))&= \inf_{s \in V(p)} d_E(q,s)= \nonumber \\
&=\min_{s^1 \in \R^{m_1}} (|q^1-s^1|^2+|q^2-p^2-P(p^1,s^1)|^2)^{1/2},
\end{align}
hence taking $s^1=q^1$,
\begin{align}\label{dEqVp}
d_E(q, V(p))& \le |q^2-p^2-P(p^1,q^1)| \le \frac{d_\infty(p,q)^2}{c^2} \le \frac{r^2}{c^2},
\end{align}
which proves the right-hand side inclusion. For the other inclusion, let $q=[q^1,q^2] \in V(p) (r^2/c_R^2) \cap B_E(p,r)$ and let $\bar{q}=[p^1, q^2-P(p^1,q^1)]$. Let $s=[s^1, p^2+P(p^1,s^1)] \in V(p)$ be the point such that $d_E(q,V(p))=d_E(q,s) \le r^2/c_R^2$ and let $\bar{s}=[s^1+p^1-q^1, p^2+P(p^1,s^1-q^1)] \in V(p)$. We have
\begin{align}\label{dEqs}
d_E(\bar{q},\bar{s} )= \left(|s^1-q^1|^2+ |q^2-P(p^1,q^1) -p^2-P(p^1,s^1-q^1)|^2 \right)^{1/2}=d_E(q,s) \le \frac{r^2}{c_R^2},
\end{align}
which implies that
\begin{align*}
d_\infty(\bar{q},\bar{s} ) \le c_R d_E(\bar{q},\bar{s} )^{1/2} \le r.
\end{align*}
On the other hand,
\begin{align*}
d_\infty(\bar{q},\bar{s} )&=\max \{ |p^1-s^1-p^1+q^1|, c | q^2-P(p^1,q^1)-p^2-P(p^1,s^1-q^1)+P(p^1,s^1+p^1-q^1)|^{1/2} \}\\
&= \max \{ |q^1-s^1|,  c| q^2-p^2-P(p^1,q^1)|^{1/2} \},
\end{align*}
hence
\begin{displaymath}
c| q^2-p^2-P(p^1,q^1)|^{1/2} \le r.
\end{displaymath}
This and the fact that $q \in B_E(p,r)$ imply that
\begin{equation*}
d_\infty(p,q)= \max \{ |p^1-q^1|, c |p^2-q^2+P(p^1,q^1)|^{1/2} \} \le r,
\end{equation*}
which completes the proof of \eqref{Vpneighb}.

\subsection{Hausdorff measures and dimensions}
Given a metric space $(X,d)$ and a real number $0<s<\infty$, the $s$-dimensional Hausdorff measure of a set $A \subset X$ is defined as
\begin{displaymath}
\mathcal{H}^s(A)=\lim_{\delta\rightarrow0} \mathcal{H}^s_\delta(A),
\end{displaymath}
where 
\begin{displaymath}
\mathcal{H}^s_\delta(A)= \inf \left\{ \sum_{j=1}^\infty \mbox{diam}_d(A_j)^s: A \subset \bigcup_{j=1}^\infty A_j, \mbox{diam}_d(A_j)< \delta \right\}
\end{displaymath}
and $\mbox{diam}_d$ denotes the diameter with respect to the metric $d$. 
The Hausdorff dimension of $A$ with respect to $d$ is then defined by
\begin{displaymath}
\dim_d(A)=\sup \{s: \mathcal{H}^s(A)>0\}.
\end{displaymath}
We will denote by $\mathcal{H}^s_E$ the $s$-dimensional Hausdorff measure with respect to the Euclidean metric, by $\mathcal{H}^s_\mathbb{G}$ the $s$-dimensional Hausdorff measure with respect to the metric $d_\infty$ and by $\dim_E$ and $\dim_\mathbb{G}$ the corresponding Hausdorff dimensions. Note that the Hausdorff dimension of a set with respect to any homogeneous metric is the same since the metrics are bi-Lipschitz equivalent.

We recall that the upper density theorem for Hausdorff measures, see for example 2.10.19 in \cite{F}, states that for every $\mathcal{H}^s $ measurable set $A \subset X$ with $\mathcal{H}^s(A)>0$, for $\mathcal{H}^s$ almost every $p \in A$,
\begin{equation}\label{upd}
2^{-s} \le \limsup_{r \rightarrow 0} \frac{\mathcal{H}^s(A \cap B_d(p,r))}{(2r)^s} \le 1
\end{equation}
and for $\mathcal{H}^s$ almost every $p \in X \setminus A$,
\begin{equation}\label{dz}
\lim_{r \rightarrow 0}\frac{\mathcal{H}^s(A \cap B_d(p,r))}{(2r)^s}=0.
\end{equation}

As explained in the introduction, Balogh, Tyson and Warhurst solved the dimension comparison problem in Carnot groups, proving that for every $A \subset \mathbb{G}$,
\begin{displaymath}
\beta_-(\dim_E A) \le \dim_\mathbb{G} A \le\beta_+( \dim_E A),
\end{displaymath}
where the functions $\beta_\pm$ were defined in \eqref{beta} for a Carnot group of step $2$. This result follows from Proposition 3.1 in the same paper \cite{BTW} stating that for any $0<R<\infty$ there exists $C_R$ such that for any $A \subset B_E(0,R)$ and $0\le s\le n$,
\begin{equation}\label{Hm}
\frac{\mathcal{H}^{\beta_+(s)}_\mathbb{G}(A) }{C_R}\le \mathcal{H}^s_E(A) \le C_R  \mathcal{H}^{\beta_-(s)}_\mathbb{G}(A).
\end{equation}

\section{Results}

We will now state and prove the theorems, whose content is a generalization to Carnot groups of step $2$ of those proved in \cite{MV} for the first Heisenberg group $\mathbb{H}^1$.

\begin{theorem}\label{thm1}
Let $0< s \le m_1$ and let $A \subset \mathbb{G}$ be such that $\mathcal{H}^s_\mathbb{G}(A) < \infty$. Then  for $\mathcal{H}^s_E$ almost every $p \in A$ there exists $0<\epsilon<1$ such that
\begin{equation*}
\liminf_{r \rightarrow 0} \frac{\mathcal{H}^s_E(A \cap B_E(p,r) \setminus V(p)(r^{1+\epsilon}))}{r^s}=0.
\end{equation*}
\end{theorem}

\begin{proof}
Since the proof proceeds essentially as that of Theorem 2 in \cite{MV}, we recall here the main steps, proving in detail only the parts that need to be changed.
We can assume that $A$ is a Borel set and that $A \subset B_E(0,R)$ for some $R>0$. Using \eqref{Hm} we can reduce to the case when there is a constant $C>0$ such that
\begin{equation}\label{red}
\frac{1}{C}\mathcal{H}^s_E(B) \le \mathcal{H}^s_\mathbb{G}(B) \le C \mathcal{H}^s_E(B)
\end{equation}
for every $B  \subset A$.
The proof proceeds by contradiction, assuming that for $\epsilon>0$ there exist $c'>0$, $0<r_0<1$ and $A'\subset A$ with $\mathcal{H}^s_E(A')>0$, such that
\begin{equation}\label{countera}
\mathcal{H}^s_E(A \cap B_E(p,r) \setminus V(p)(Kr^{1+\epsilon})) > c'r^s
\end{equation}
for every $p \in A'$ and $0<r<r_0$, where $K=4/c^2+1$. We will get a contradiction at the end of the proof letting $\epsilon$ be small enough depending only on $s$, $R$ and $C$. We can also assume using \eqref{upd} that
\begin{equation}\label{a}
\mathcal{H}^s_E(A \cap B_E(p,r)) \le 3^s r^s \quad \mbox{and} \quad \mathcal{H}^s_\mathbb{G}(A' \cap B_H(p,r)) > r^s/(2C)
\end{equation}
for $r$ small enough so that $r^2 << r^{1+\epsilon}$. Let $k \in \mathbb{N}$ be such that $r^{(1+\epsilon)^{k+1}} \le r^2 < r^{(1+\epsilon)^k}$.
Using the $5r$-covering theorem and \eqref{a}, for $j=1, \dots, k$ we can find $p_{j,i} \in A' \cap B_\infty(p,r)$ such that
\begin{equation*}
A'  \cap B_\infty(p,r)= \bigcup_{i=1}^{M_j} A'  \cap B_\infty(p,r) \cap B_E(p_{j,i} ,5  r^{(1+ \epsilon)^j}),
\end{equation*}
where the balls $B_E(p_{j,i}, r^{(1+\epsilon)^j})$, $i=1, \dots, M_j$, are disjoint and $M_j \ge r^{s(1-(1+\epsilon)^j)}/(15^s2C)$ (here we use \eqref{a}).

Now we show that the sets
\begin{equation}\label{disj}
\bigcup_{i=1}^{m_j} B_E(p_{j,i},  r^{(1+\epsilon)^j}) \setminus V(p_{j,i})(K r^{(1+\epsilon)^{j+1}}), \ j=1, \dots,k,
\end{equation}
are disjoint. We show the details since they are a bit different from the case of $\mathbb{H}^1$.
Let $j \in \{1, \dots, k-1\}$ and let $B_E(p_{j,i}, r^{(1+\epsilon)^j})$ and $B_E(p_{m,l}, r^{(1+\epsilon)^m})$, $ j +1 \le m \le k$, $i\in \{1, \dots, M_j\}, l \in \{1, \dots, M_m\}$, be such that $B_E(p_{j,i},r^{(1+\epsilon)^j} ) \cap B_E(p_{m,l}, r^{(1+\epsilon)^m}) \neq \emptyset$. We want to show that
\begin{equation}\label{intB}
( B_E(p_{j,i}, r^{(1+\epsilon)^j}) \setminus V(p_{j,i})(K r^{(1+\epsilon)^{j+1}})) \cap B_E(p_{m,l}, r^{(1+\epsilon)^m})= \emptyset.
\end{equation}
By the same calculation used to obtain \eqref{dEqVp} we have
\begin{align*}
d_E(p_{m,l}, V(p_{j,i}))  \le \frac{d_\infty(p_{m,l},p_{j,i})^2}{c^2}\le \left( \frac{d_\infty(p_{m,l},p)+d_\infty(p,p_{j,i})}{c}\right)^2\le \frac{4r^2}{c^2},
\end{align*}
which implies that for every $q \in B_E(p_{m,l}, r^{(1+\epsilon)^m})$
\begin{align*}
d_E(q,V(p_{j,i}))\le \frac{4r^2}{c^2}+r^{(1+\epsilon)^m} \le K r^{(1+\epsilon)^{j+1}}
\end{align*}
since $r^2 < r^{(1+\epsilon)^{k}}\le r^{(1+\epsilon)^{j+1}}$. Thus \eqref{intB} holds. Then, as in the case of $\mathbb{H}^1$, we get by \eqref{countera}
\begin{align*}
\mathcal{H}^s_E(A \cap B_E(p,3r))& \ge \sum_{j=1}^k \sum_{i=1}^{M_j} \mathcal{H}^s_E(A \cap B_E(p_{j,i},  r^{(1+\epsilon)^j}) \setminus V(p_{j,i})(K r^{(1+\epsilon)^{j+1}}))\\
& \ge c' \sum_{j=1}^k M_j r^{s(1+\epsilon)^j} \ge \frac{1}{15^s2C} \sum_{j=1}^k r^{s(1-(1+\epsilon)^j)} r^{s(1+\epsilon)^j} \ge \frac{k}{15^s2C}  r^s.
\end{align*}
Since $k \ge \log2/(2\log(1+\epsilon))$, this is a contradiction with \eqref{upd} when $\epsilon$ is small enough.
\end{proof}

\begin{theorem}\label{sgreat}
Let $s \ge m_2$ and $A \subset \mathbb{G}$ be such that $\mathcal{H}^{s}_E(A) <\infty$. Then for $\mathcal{H}^{s+m_2}_\mathbb{G}$ almost every $p \in A$ there exists $\delta >0$ such that
\begin{equation}\label{claim2}
\limsup_{r \rightarrow 0} \frac{\mathcal{H}^s_E (A \cap B_E(p,r) \setminus V(p)(\delta r))}{(2r)^s} > \frac{1}{2^{s+1}}.
\end{equation}
\end{theorem}
\begin{proof}
We follow essentially the proof of Theorem 3 in \cite{MV}, recall the main steps and show the details to prove \eqref{Vp} since they are different. We can reduce to the case when $A \in B_E(0,R)$ is a Borel set and there is $C>0$ such that
\begin{equation}\label{red2}
\frac{1}{C}\mathcal{H}^s_E(B) \le \mathcal{H}^{s+m_2}_\mathbb{G}(B) \le C \mathcal{H}^s_E(B)
\end{equation}
for every $B  \subset A$. Moreover, we can assume that there exists $0<r_0<1$ such that for every $0<r<r_0$ and almost every $p \in A$,
\begin{equation}\label{measball}
\mathcal{H}^s_E(A \cap B_E(p,r)) \le 3^sr^s \quad \mbox{and} \quad \mathcal{H}^{s+m_2}_\mathbb{G}(A \cap B_\infty(p,r))\le 3^{s+m_2} r^{s+m_2}.
\end{equation}
More precisely, we would need to split $A$ into a countable union of sets on which $r_0$, which depends on the point $p$, is essentially $2^{-j}$, $j \in \mathbb{N}$, but we skip these details. Let 
\begin{equation}\label{delta}
0<\delta< \frac{c^2}{c_R^2(2^{s+5m_2+1}3^{s+m_2}C)^{1/{m_2}}}
\end{equation}
and note that it depends only on $s$, $R$ and $C$.

We now show that for every $0<r<\min\{r_0,4 (c_R/c)^2\delta\}$ there exist $p_1, \dots, p_k$, $k \approx (\delta/r)^{m_2}$, such that
\begin{equation}\label{Vp}
B_E(p,r) \cap V(p)(\delta r) \subset \bigcup_{i=1}^k B_\infty(p_i, 2r).
\end{equation}
To prove it, let $L(p)=\{[q^1,q^2] \in \mathbb{G}: q^1=p^1\}$. If $q \in L(p)$ and $d_E(q, V(p)) \le \delta r$, then $|q^2-p^2| \le (c_R/c)^2 \delta r$. To see this, note that there exists $s=[s^1,p^2+P(p^1,s^1)] \in V(p)$ such that $d_E(q,s) \le \delta r$. This implies by \eqref{compd} that $d_\infty(q,s) \le c_R (\delta r)^{1/2}$ and since
\begin{align*}
d_\infty(q,s)= &\max \{ |p^1-s^1|, c |q^2-p^2-P(p^1,s^1)+P(p^1,s^1)|^{1/2} \} \\ 
&= \max \{ |p^1-s^1|, c |q^2-p^2|^{1/2}\},
\end{align*}
we have $c |q^2-p^2|^{1/2} \le c_R (\delta r)^{1/2}$, that is $ |q^2-p^2| \le (c_R/c)^2 \delta r$.
Let us cover the Euclidean ball with center $p^2$ and radius $(c_R/c)^2\delta r$ in $L(p) \subset \mathbb{R}^{m_2}$ with Euclidean balls
\begin{equation}\label{k}
B_E(p_j^2,r^2), \quad j=1, \dots, k, \quad k \le \left(\frac{2c_R}{c}\right)^{2m_2} \left(\frac{\delta r}{r^2}\right)^{m_2}
\end{equation} 
so that the corresponding balls of radii $r^2/2$ are disjoint. Let $p_j=[p^1,p_j^2] \in L(p)$. To see that \eqref{Vp} holds, let $q=[q^1,q^2] \in B_E(p,r) \cap V(p)(\delta r)$ and let $\bar{q}=[p^1, q^2-P(p^1,q^1)] \in L(p)$. Then, proceeding as in the proof of \eqref{dEqs}, we also have $d_E(\bar{q}, V(p)) \le \delta r$. 

Since $\bar{q} \in L(p) \cap V(p)(\delta r)$, which is contained in the Euclidean ball with center $p^2$ and radius $(c_R/c)^2\delta r$ in $L(p)$, there exists $j \in \{1, \dots k \}$ such that $d_E(\bar{q}, p_j)=|q^2-P(p^1,q^1)- p_j^2| \le r^2$. Hence
\begin{align*}
d_\infty(p_j,q)= \max \{ |p^1-q^1|, c | p_j^2-q^2+P(p^1,q^1)|^{1/2}\} \le \max \{r, cr\} \le r,
\end{align*}
which means that $q \in B_\infty(p_j,r)$. Thus \eqref{Vp} holds.

Hence by \eqref{Vp}, \eqref{red2}, \eqref{measball}, \eqref{k} and \eqref{delta} we have for every $0<r<r_0$
\begin{align*}
\mathcal{H}^s_E( A \cap B_E(p,r) \cap V(p)(\delta r)) &\le \sum_{i=1}^k \mathcal{H}^s_E(A \cap B_\infty(p_i,2r)) \\
& \le C \sum_{i=1}^k \mathcal{H}^{s+m_2}_\mathbb{G}(A \cap B_\infty(p_i,2r)) \\
& \le C  k 3^{s+m_2}2^{s+m_2} r^{s+m_2}\\
& \le  C \left(\frac{2c_R}{c}\right)^{2m_2} 3^{s+m_2}2^{s+3m_2}\left(\frac{\delta}{r}\right)^{m_2}  r^{s+m_2}  < \frac{1}{2}r^s.
\end{align*}
Thus for $\mathcal{H}^s_E$ almost every $p \in A$ there exists $\delta>0$ such that
\begin{equation*}
\limsup_{r\rightarrow 0} \frac{\mathcal{H}^s_E( A \cap B_E(p,r) \cap V(p)(\delta r)) }{(2r)^s}<\frac{1}{2^{s+1}},
\end{equation*}
which proves \eqref{claim2} by \eqref{upd} and \eqref{dz}.
\end{proof}

\begin{theorem}\label{ssmall1}
Let $0 < s<m_2$ and $A \subset \mathbb{G}$ be such that $\mathcal{H}^s_E(A)<\infty$. Then  for $\mathcal{H}^{2s}_\mathbb{G}$ almost every $p \in A$ there exists $\delta>0$ such that
\begin{equation}\label{claim3}
\limsup_{r \rightarrow 0} \frac{\mathcal{H}^s_E (A \cap B_E(p,r) \setminus V(p)(\delta r))}{(2r)^s} > 0.
\end{equation}
\end{theorem}

\begin{proof}
The proof proceeds as in the case of $\mathbb{H}^1$ with small changes (see Theorem 4 in \cite{MV}), but we show it here for completeness.
We may assume that $A$ is a Borel, $A\subset B_E(0,R)$ for some $R>0$ and there exists $C>0$ such that 
\begin{equation}\label{Hs2s}
\frac{1}{C} \mathcal{H}^s_E(B) \le \mathcal{H}^{2s}_\mathbb{G}(B) \le C  \mathcal{H}^s_E(B)
\end{equation}
for every $B\subset A $.

The proof proceeds by contradiction. Assume that we can find a Borel set $A_0 \subset A$ such that $\mathcal{H}^s_E(A_0)>0$ and \eqref{claim3} fails for $p\in A_0$ for every $\delta>0$. Fix $\delta>0$ and $\epsilon>0$, which will be chosen sufficiently small at the end of the proof. Then there exist a Borel set $A' \subset A_0$ and $r_0>0$ such that $\mathcal{H}^s_E(A')>\mathcal{H}^s_E(A_0)/2$ and for every $p \in A'$ and for every $0<r<r_0$, 
\begin{equation}\label{counterass}
\mathcal{H}^s_E (A \cap B_E(p,r) \setminus V(p)(\delta r)) <\epsilon r^s.
\end{equation}
Let $0<\eta<\min \{r_0,\delta\}$. By \eqref{upd} for $\mathcal{H}^s_E$ almost all $p \in A'$ there exists $r_p<\eta$ such that 
\begin{equation}\label{rs}
\frac{r_p^s}{3^s} \le \mathcal{H}^s_E(A' \cap B_E(p,r_p)) \le 3^s r_p^s.
\end{equation} 
We apply Vitali's covering theorem (see e.g. Theorem 2.8 in \cite{M}) to the family $\{ B_E(p,r_p): p \in A' \ \mbox{such that } r_p \mbox{ exists}\}$ to find a subfamily of disjoint balls, $\{B_E(p_i,r_i)\}_{i=1}^\infty$, such that
\begin{equation}
\mathcal{H}^s_E\left( A' \setminus \bigcup_{i=1}^\infty B_E(p_i,r_i) \right)=0.
\end{equation}
This implies by \eqref{rs} that
\begin{align}\label{HsEAR}
\mathcal{H}^s_E(A') = \mathcal{H}^s_E\left(A'\cap \bigcup_{i=1}^\infty B_E(p_i,r_i)\right)=\sum_{i=1}^\infty  \mathcal{H}^s_E(A' \cap  B_E(p_i,r_i)) \ge  \frac{1}{3^s} \sum_{i=1}^\infty r_i^s.
\end{align}
Since $p_i \in B_E(0,R)$, we have by \eqref{compd} that 
\begin{equation}\label{change}
\mbox{diam}_\infty(B_E(p_i,r_i))\le c_R \mbox{diam}_E(B_E(p_i,r_i))^{1/2}\le c_R (2\eta)^{1/2}.
\end{equation}
Let $\eta'= c_R (2\eta)^{1/2}$. Then 
\begin{align}\label{H2sA}
\mathcal{H}^{2s}_{\mathbb{G},\eta'} (A')  \le&   \sum_{i=1}^\infty  \mathcal{H}^{2s}_{\mathbb{G},\eta'}(A' \cap  B_E(p_i,r_i)) \nonumber \\
\le &  \sum_{i=1}^\infty  \mathcal{H}^{2s}_{\mathbb{G},\eta'}(A' \cap  B_E(p_i,r_i) \cap V(p_i)(\delta r_i))\\ &+\sum_{i=1}^\infty  \mathcal{H}^{2s}_{\mathbb{G},\eta'}(A' \cap  B_E(p_i,r_i)\setminus V(p_i)(\delta r_i)).\nonumber
\end{align}
Moreover, we claim that
\begin{align}\label{diamH}
 \mathcal{H}^{2s}_{\mathbb{G},\eta'}(A'\cap  B_E(p_i,r_i) \cap V(p_i)(\delta r_i)) &\le \mbox{diam}_\infty(B_E(p_i,r_i) \cap V(p_i)(\delta r_i))^{2s} \nonumber \\ & \le (2( c_R+2) (\delta r_i)^{1/2})^{2s}= C'' (\delta r_i)^s. 
\end{align}
To see this, let $q, q' \in B_E(p_i,r_i) \cap V(p_i)(\delta r_i)$ and let $\bar{q}, \bar{q}' \in  V(p_i)$ be such that $d_E(q,\bar{q}) \le \delta r_i$ and $d_E(q',\bar{q}') \le \delta r_i$. Then
\begin{align*}
d_\infty(q,q') \le d_\infty(q,\bar{q})+d_\infty(\bar{q},\bar{q}')+d_\infty(\bar{q}',q').
\end{align*}
We have by \eqref{compd}
\begin{equation}\label{change2}
d_\infty(q,\bar{q}) \le c_Rd_E(q,\bar{q})^{1/2} \le c_R(\delta r_i)^{1/2}, \quad d_\infty(q',\bar{q}') \le c_R d_E(q',\bar{q}')^{1/2} \le c_R(\delta r_i)^{1/2}
\end{equation} 
and by \eqref{dVp}, since $\bar{q}, \bar{q}' \in V(p_i)$,
\begin{equation*}
d_\infty(\bar{q},\bar{q}') \le d_\infty(\bar{q},p_i)+d_\infty(p_i, \bar{q}') \le d_E(\bar{q},p_i)+d_E(p_i, \bar{q}') \le 4r_i.
\end{equation*}
Since $r_i < \eta < \delta$, it follows that $r_i^{2} < \delta r_i$, which implies $r_i \le (\delta r_i)^{1/2}$. Hence
\begin{equation*}
d_\infty(q,q') \le 2 c_R (\delta r_i)^{1/2} +4 r_i \le 2(c_R+2) (\delta r_i)^{1/2},
\end{equation*}
which proves \eqref{diamH}.
On the other hand, by \eqref{Hs2s} and \eqref{counterass} we have
\begin{align*}
\mathcal{H}^{2s}_{\mathbb{G}}(A' \cap  B_E(p_i,r_i)\setminus V(p_i)(\delta r_i)) \le C \mathcal{H}^{s}_{E}(A' \cap  B_E(p_i,r_i)\setminus V(p_i)(\delta r_i))<C \epsilon r_i^s,
\end{align*}
thus also
\begin{equation}\label{minus}
\mathcal{H}^{2s}_{\mathbb{G},\eta'}(A' \cap  B_E(p_i,r_i)\setminus V(p_i)(\delta r_i))< C \epsilon r_i^s.
\end{equation}
Hence \eqref{H2sA}, \eqref{diamH}, \eqref{minus} and \eqref{HsEAR} imply that
\begin{align*}
\mathcal{H}^{2s}_{\mathbb{G},\eta'} (A') \le (C''\delta^s+C\epsilon) \sum_{i=1}^\infty r_i^s \le 3^s ( C''\delta^s+C\epsilon) \mathcal{H}^s_E(A')
\le  3^s 2( C''\delta^s+C\epsilon) \mathcal{H}^s_E(A_0).
\end{align*}
Letting $\eta$ (thus also $\eta'$) go to $0$, 
\begin{align*}
0<\mathcal{H}^{2s}_{\mathbb{G}} (A_0) < 2\mathcal{H}^{2s}_{\mathbb{G}} (A') \le 3^s 2 ( C''\delta^s+C\epsilon) \mathcal{H}^s_E(A_0).
\end{align*}
Choosing $\delta$ and $\epsilon$ arbitrarily small, we get a contradiction, which completes the proof.
\end{proof}

\section{Examples}

We now present some examples that show in which sense Theorem \ref{thm1} is sharp. Examples \ref{ex1s1}, \ref{ex2s1}, \ref{ex1} and \ref{ex2} show that the conclusion of Theorem \ref{thm1} does not hold for any arbitrary small neighbourhood. In particular, we cannot replace $\liminf$ by $\limsup$ or the $r^{1+\epsilon}$-neighbourhood of $V(p)$ by the $r^2$-neighbourhood, which essentially behaves like the ball $B_\infty(p,r)$ inside $B_E(p,r)$ by \eqref{Vpneighb}. The first two examples in Section \ref{exs1} are $1$-dimensional sets in any Carnot group of step $2$, whereas the other two in Section \ref{sm1h} are $m$-dimensional sets in the Heisenberg group $\mathbb{H}^m$. Note that $\mathbb{H}^m$ can be identified with $\R^{2m+1}$, where $m_1=2m$, hence these sets have dimension $m_1/2$. These four examples generalize those of $\mathbb{H}^1$, see Examples 5 and 6 in \cite{MV}. Example \ref{ex3} in Section \ref{sbig} shows that in any Carnot group of step $2$ sets of Euclidean Hausdorff dimension $s$ greater than $m_1$ and homogeneous dimension $\beta_-(s)=2s-m_1$ can have positive density also far from horizontal planes.

\subsection{Examples for $s=1$}\label{exs1}

The following examples can be constructed in any Carnot group $\mathbb{G}$ of step $2$ following the same procedure used in $\mathbb{H}^1$ for Examples 5 and 6 in \cite{MV}. 

\begin{ex}\label{ex1s1}
There exists a compact set $F \subset \mathbb{G}$ such that for some positive constant $C$, $\mathcal{H}^1_\mathbb{G}(F)>0$ and $\mathcal{H}^1_\mathbb{G}(A) \leq  C\mathcal{H}^1_E(A)<\infty$ for $A\subset F$, and for $p \in F$, 
\begin{equation*}
\limsup_{r \rightarrow 0} \frac{\mathcal{H}^1_E (F \cap B_E(p,r) \setminus V(p)(r/8))}{2r} \geq \frac{1}{8}.
\end{equation*}
\end{ex}

\begin{ex}\label{ex2s1}
For any $M, 1<M<\infty$, there exists a compact set $F \subset \mathbb{G}$  such that for some positive constant $C$, $\mathcal{H}^1_\mathbb{G}(F)>0$ and $\mathcal{H}^1_\mathbb{G}(A) \leq  C\mathcal{H}^1_E(A)<\infty$ for $A\subset F$, and for $p \in F$, 
\begin{equation*}
\liminf_{r \rightarrow 0} \frac{\mathcal{H}^1_E (F \cap B_E(p,r) \setminus V(p)(Mr^2))}{2r} \geq \frac{1}{16}.
\end{equation*}
\end{ex}

The sets in the examples will be subsets of the $(x_1,x_{m_1+1})$-plane
\begin{align}\label{V}
V_{1,m_1+1}=\{& p=(x_1, \dots, x_{m_1},x_{m_1+1}, \dots, x_n) \in \mathbb{G}: \nonumber \\& x_2=\dots = x_{m_1}=0, x_{m_1+2}=\dots =x_n=0\}.
\end{align}
Note that for $p, q \in V_{1,m_1+1}$, we have $P(p^1,q^1)=0$ by \eqref{Pj} and we will denote the points by $p=(x_1,x_{m_1+1})$, $q=(y_1,y_{m_1+1})$. It follows that
\begin{displaymath}
d_\infty(p,q)=\max \{ |x_1-y_1|, c|x_{m_1+1}-y_{m_1+1}|^{1/2} \}
\end{displaymath}
and 
\begin{displaymath}
V(p) \cap V_{1,m_1+1}=\{ q=(y_1,y_{m_1+1}) \in V_{1,m_1+1}: y_{m_1+1}=x_{m_1+1} \}.
\end{displaymath}
The sets in Examples \ref{ex1s1} and \ref{ex2s1} can now be constructed with the same procedure used in $\mathbb{H}^1$, starting with the unit square $[0,1]^2 \subset V_{1,m_1+1}$, taking certaing subrectangles and iterating the construction. In the next section we will use the same procedure in higher dimensions and so we do not repeat it here. In this case, since we are in the plane, it would be exactly the same as the one used in \cite{MV}, which corresponds to $m=1$ in the next section.

\subsection{Examples for $s=m_1/2$ in Heisenberg groups}\label{sm1h}

We now construct similar examples to those in the previous section but having dimension half the dimension of the first layer and only in Heisenberg groups.
As briefly explained in Section \ref{step2prel}, we denote points of $\mathbb{H}^m $ by $p=(x,y,t) \in \R^{2m+1}$ and the group operation was defined in \eqref{operHn}. A homogeneous metric that is commonly used in $\mathbb{H}^m$ is the Kor\'anyi one, defined by
\begin{displaymath}
d_H(p,q)=((|x-x'|^2+|y-y'|^2)^2+(t-t'+2(\langle x',y \rangle - \langle x, y' \rangle)^2)^{1/4},
\end{displaymath}
where $q=(x',y',t')$.
We denote by $\mathcal{H}^m_H$ the $m$-dimensional Hausdorff measure with respect to $d_H$. The horizontal hyperplane $V(p)$ passing through a point $p=(x,y,t)$ consists of those points $q=(x',y',t') \in \mathbb{H}^m$ such that
\begin{equation}\label{VpHn}
t'=t+2(\langle x',y \rangle - \langle x, y'\rangle).
\end{equation}

\begin{ex}\label{ex1}
There exists a compact set $F \subset \mathbb{H}^m$ such that $\mathcal{H}^m_H(F)>0$, $\mathcal{H}^m_H(A) \leq  C\mathcal{H}^m_E(A)<\infty$ for any $A\subset F$ for some constant $C>0$, and for $p \in F$, 
\begin{equation}\label{claim4} 
\limsup_{r \rightarrow 0} \frac{\mathcal{H}^m_E (F \cap B_E(p,r) \setminus V(p)(r/8))}{(2r)^m} \geq \left( \frac{1}{8} \right)^m.
\end{equation}
\end{ex}

\begin{ex}\label{ex2}
For any $M$, $1<M<\infty$, there exists a compact set $F \subset \mathbb{H}^m$  such that $\mathcal{H}^m_H(F)>0$, $\mathcal{H}^m_H(A) \leq  C\mathcal{H}^m_E(A)<\infty$ for $A\subset F$ for some constant $C>0$, and for $p \in F$, 
\begin{equation}\label{claim5}
\liminf_{r \rightarrow 0} \frac{\mathcal{H}^m_E (F \cap B_E(p,r) \setminus V(p)(Mr^2))}{(2r)^m} \geq \left( \frac{1}{2^5} \right)^m.
\end{equation}
\end{ex}

Both examples are based on the following construction, which is a simple generalization of the one used in Examples 5 and 6 \cite{MV}.
Let
\begin{equation*}
V_{x,t}=\{(x,y,t) \in \mathbb{H}^m: y=0\} \subset \R^{m+1}.
\end{equation*}
In both examples the set $F$ will be a subset of $V_{x,t}$. We denote the points of $V_{x,t}$ by $p=(x,t)$.
For $p=(x,t), q=(x',t') \in V_{x,t}$ we have
\begin{align*}
d_H(p,q)=(|x-x'|^4+|t-t'|^2)^{1/4}.
\end{align*}
For $p=(x,t) \in V_{x,t}$, \eqref{VpHn} implies that
\begin{align*}
V(p) \cap V_{x,t}=\{ (x',t'): t'=t\}.
\end{align*}

\begin{rem}
The fact that $V(p) \cap V_{x,t}$ is parallel to the $x$-coordinates plane simplifies significantly the construction of Examples \ref{ex1} and \ref{ex2}. In a step $2$ Carnot group $\mathbb{G}$ we cannot in general have this simplification, since it would correspond to finding a certain $V'\subset \mathbb{G}$ such that 
\begin{equation}\label{cond}
P(p^1,q^1)=0 \quad \mbox{for all} \quad p, q \in V',
\end{equation}
which is what happens in $\mathbb{H}^m$ when $V'=V_{x,t}$. This holds, for example, when the points in $V'$ have all coordinates of the first layer equal to zero except for $1$, as in the case of $V_{1,m_1+1}$ defined in \eqref{V}. In this case, however, we would have $V' \subset \R^{1+m_2}$. We were not able to find another easy way to ensure that \eqref{cond} would hold if, for example, only half of the coordinates of $\R^{m_1}$ were zero, as in the case of $V_{x,t}$ in $\mathbb{H}^m$. For this reason we construct the examples corresponding to $s=m_1/2$ only in $\mathbb{H}^m$.
\end{rem}

Given $L>0$, an integer $N\ge1$, $0<\lambda\le \frac{1}{2}$ and the $(m+1)$-dimensional parallelepiped
\begin{displaymath}
R=[a_1,b_1] \times \dots \times [a_m,b_m] \times [c,d] \subset V_{x,t}
\end{displaymath}
such that $b_i-a_i=L$ for every $i=1, \dots, m$ and $\lambda L\le (d-c)/2$, we choose $(2N)^m$ subsets of $R$, which are $(m+1)$-dimensional parallelepipeds, in the following way. We divide $\prod_{i=1}^m [a_i,b_i]$ into $(2N)^m$ identical $m$-dimensional cubes $C_j$, $j=1, \dots, (2N)^m$, of side length $L/2N$. We choose $R_1=R^c_1=C_1\times [c,c+\lambda L]$. Then for each $j$ we define either $R_j=R_j^c=C_j \times [c,c+\lambda L]$ or $R_j=R_j^d=C_j \times[d-\lambda L, d]$ in such a way that two distinct $R_i^c$, $R_j^c$ do not share an $m$-dimensional face (and the same for $R_i^d, R_j^d$). Hence we have $(2N)^m/2$ sub-parallelepipeds of the form $R_i^c$ and $(2N)^m/2$ of the form $R_i^d$. Then we define $\mathcal{R}(R,N,\lambda)$ as the collection of these $(2N)^m$ parallelepipeds and we have that 
\begin{align}\label{projx}
\pi_x\left(\bigcup_{R \in \mathcal{R}(R,N,\lambda)} R \right)= \prod_{i=1}^m [a_i,b_i],
\end{align} 
where $\pi_x(x,t)=x$.

The case $m=1$ corresponds to the construction in Examples 5 and 6 in \cite{MV}.
Let us write down explicitly the case $m=2$. Given $R=[a_1,b_1] \times [a_2,b_2]\times [c,d]$, with $b_i-a_i=L$, an integer $N\ge1$ and $0<\lambda\le \frac{1}{2}$ such that $\lambda L \le (d-c)/2$, $\mathcal{R}(R,N,\lambda)$ is defined as the following collection of $4N^2$ sub-parallelepipeds of $R$.

\begin{itemize}
\item For $i,j=0,1, \dots, N-1$, let \begin{align*}
R^c_{i,j}=\left[a_1+2 i \frac{L}{2N},a_1+2 i \frac{L}{2N}+\frac{L}{2N}\right] \times \left[a_2+2 j \frac{L}{2N},a_2+2 j \frac{L}{2N}+\frac{L}{2N}\right]\times [c,c+\lambda L]
\end{align*}
and for $l,k=0,1, \dots, N-1$, \begin{align*}
R^c_{l,k}=&\left[a_1+(2 l+1) \frac{L}{2N},a_1+(2 l+1) \frac{L}{2N}+\frac{L}{2N}\right]\\ &\times \left[a_2+(2 k+1) \frac{L}{2N},a_2+(2 k+1) \frac{L}{2N}+\frac{L}{2N}\right]\times [c,c+\lambda L].
\end{align*}
Thus we have $2N^2$ sub-parallelepipeds of this form.
\item For $i,j=0,1, \dots, N-1$, let \begin{align*}
R^d_{i,j}=&\left[a_1+2 i \frac{L}{2N},a_1+2 i \frac{L}{2N}+\frac{L}{2N}\right] \times \left[a_2+(2 j+1) \frac{L}{2N},a_2+(2 j+1) \frac{L}{2N}+\frac{L}{2N}\right]\\ &\times [d-\lambda L,d]
\end{align*}
and  for $l,k=0,1, \dots, N-1$,\begin{align*}
R^d_{l,k}=&\left[a_1+(2 l+1) \frac{L}{2N},a_1+(2 l+1) \frac{L}{2N}+\frac{L}{2N}\right] \times \left[a_2+2 k \frac{L}{2N},a_2+2 k \frac{L}{2N}+\frac{L}{2N}\right]\\ &\times [d-\lambda L,d].
\end{align*}
Thus we have $2N^2$ sub-parallelepipeds of this form.
\end{itemize}

Let us now construct the sets in Examples \ref{ex1} and \ref{ex2} for any $m \ge 1$.
Given a sequence of integers $(N_k)_k$, $N_k\geq 1$, and a sequence of positive numbers $(\lambda_k)_k$, $\lambda_k\leq 1/2$, we define for $k\geq 1$, 

$$\mathcal R_0 = \mathcal R([0,1]^{m+1},1,1/2),$$
$$\mathcal R_{k} = \bigcup_{R\in\mathcal R_{k-1}}\mathcal R(R,N_k,\lambda_k),$$
and 
$$F = \bigcap_{k=0}^{\infty}\bigcup_{R\in\mathcal R_k}R.$$

It follows that $F\subset V_{x,t}$ is compact and $\pi_x(F)=[0,1]^m$ by \eqref{projx}, thus both $\mathcal{H}^m_E(F)$ and $\mathcal{H}^m_H(F)$ are at least $1$. We can also check, using the natural coverings with the parallelepipeds of $\mathcal R_k$, that these measures are finite if $\lambda_k$ goes to $0$ sufficiently fast. To see this, let $h_k$ be the length of the sides in the directions of the $x$-coordinates of the parallelepipeds of $\mathcal R_k$ (all sides in the directions of the $x$-coordinates have the same length) and let $v_k$ be the length of their vertical sides (in the $t$-direction). Note that 
\begin{displaymath}
h_0=\frac{1}{2},\ h_k=\frac{h_{k-1}}{2N_k} \ \mbox{for } k \ge 1, \quad v_0=\frac{1}{2},\ v_k=\lambda_k h_{k-1}\ \mbox{for } k \ge 1.
\end{displaymath}
For every $R\in \mathcal R_k$, we have
\begin{displaymath}
\diam_E(R) = (m h_k^2+v_k^2)^{1/2} \quad \mbox{and} \quad \diam_H(R) = (m^2 h_k^4+v_k^2)^{1/4}.
\end{displaymath}
If $v_k/h_k$ tends to zero as $k\to\infty$, then for $R\in\mathcal R_k$,
\begin{equation}\label{eq7}
h_k^m \le \mathcal{H}^m_E(F\cap R) \le (\sqrt{m} h_k)^m,
\end{equation}
which implies 
\begin{displaymath}
1 \le \mathcal{H}^m_E(F)\le (\#\mathcal R_{k}) (\sqrt{m} h_k)^m= (2 N_k)^m (\sqrt{m} h_k)^m  \le  m^{m/2}.
\end{displaymath}
If moreover, $v_k \le Ch_k^2$ for some $0<C<\infty$ for all large enough $k$, then for $A\subset F$,
\begin{equation}\label{eq1}
\mathcal{H}^m_H(A)\leq \frac{(m^2+C^2)^{1/4}}{\sqrt{m}} \mathcal{H}^m_E(A) .
\end{equation}
In Example \ref{ex1} we will choose $v_k = h_k^2$ and in Example \ref{ex2} $v_k =136 Mm h_k^2$ for large $k$, hence the above conditions on $h_k$ and $v_k$ will be satisfied. 

For Example \ref{ex1} let $N_k = 2^{^{2^{k-1}-1}}$ and $\lambda_k = 2^{-3\cdot 2^{k-1}}$. It follows that $h_k=2^{^{-2^k}}$ and $v_k=2^{^{-2^{k+1}}}$ for $k\ge 1$. Thus  $h_{k+1}=v_k$, which means that for $R\in\mathcal R_k$, the horizontal sides of each parallelepiped $R'$ of $\mathcal R_{k+1}$ inside $R$ have the same length as the vertical sides of $R$. This implies that for $p\in R'$ and $r_k=4h_{k+1}$, 
$B_E(p,r_k) \setminus V(p)(r_k/8)$ contains another parallelepiped $R''$ of $\mathcal R_{k+1}$. Indeed, if $R'$ is of the form $R^c_i$ for some $i$ then there is one $R^d_j$ that is at least at vertical distance $v_k- 2v_{k+1}>v_k/2=r_k/8$ from $p$, so it is outside $V(p)(r_k/8)$, and is contained in $B_E(p,r_k)$ (and similarly if $R'$ is of the form $R^d_i$ we can find $R''$ of the form $R^c_j$).
Hence, using \eqref{eq7},
$$\frac{\mathcal{H}^m_E (F \cap B_E(p,r_k) \setminus V(p)(r_k/8))}{(2r_k)^m} \geq \frac{\mathcal{H}^m_E (F \cap R'')}{(8h_{k+1})^m} =\left( \frac{1}{8} \right)^m,$$
from which \eqref{claim4} follows. Recalling also \eqref{eq1}, we have completed Example \ref{ex1}.

For Example \ref{ex2}, given $1<M<\infty$, we let $N_k = 1$ for all $k \ge1$ and $\lambda_k =1/2$ when $68Mm4^{-k}\geq 2^{-k}$, that is, $2^k\leq 68Mm$, and $\lambda_k=68Mm2^{-k-1}$ for all larger $k$. Thus, for large enough $k$ so that $2^k > 68Mm$, we have $h_k=2^{-k-1}$ and $v_k=34Mm4^{-k}$. For $R\in\mathcal R_k$, there are $2^m$ sub-parallelepipeds $R_j$ of $\mathcal R_{k+1}$ inside $R$, of which half are of the form $R_j^d$ and half of the form $R_i^c$. The distance between any $R_j^d$ and any $R_i^c$ is at least $v_k-2v_{k+1}=34Mm4^{-k}-2\cdot 34Mm4^{-k-1}=17Mm4^{-k}$.

Given $0<r<1$, let $k$ be such that $\sqrt{m}2^{1-k}\leq r < \sqrt{m}2^{2-k}$. We may assume that $r$ is small enough so that $2^k > 68Mm$. Let $R, R_j^d$ and $R_i^c$ be as in the previous paragraph and $p\in R_j^d$. Since $Mr^2 \leq Mm4^{2-k}< 17Mm4^{-k}$, all the parallelepipeds $R_i^c$'s lie outside $V(p)(Mr^2)$. On the other hand, note that $\sqrt{m}h_k= \sqrt{m}2^{-k-1} \le r/2$ and $v_k= 34Mm4^{-k} \le  r/2$ because $2^k > 68Mm$. Hence we have
\begin{align*}
 \mbox{diam}_E(R) \le \sqrt{mh_k^2+v_k^2} \le \sqrt{\frac{r^2}{4}+\frac{r^2}{4}}\le r,
\end{align*}
which means that each $R_i^c$ lies inside $B_E(p,r)$. 
This implies that $B_E(p,r) \setminus V(p)(Mr^2)$ contains (at least one) $R_i^c$, thus
\begin{align*}
\frac{\mathcal{H}^m_E (F \cap B_E(p,r) \setminus V(p)(Mr^2))}{(2r)^m} \ge \frac{\mathcal{H}^m_E(F \cap R_i^c)}{(2r)^m}\ge\left( \frac{\sqrt{m}2^{-k-2}}{\sqrt{m}2^{3-k}}\right)^m=  \left(\frac{1}{2^{5}}\right)^m,
\end{align*}
which is what we wanted to show in Example \ref{ex2}.

\subsection{Example for $m_1\le s \le n$}\label{sbig}

The last example that we present shows that in any Carnot group $\mathbb{G}$ of step $2$ the conclusion of Theorem \ref{thm1} does not hold when $s \ge m_1$.

We will use the following estimate. For any $p^1=(x_1, \dots, x_{m_1})$, $q^1=(y_1, \dots, y_{m_1})$, 
\begin{align}\label{P}
|P(p^1,q^1)|^2 &= \sum_{j=m_1+1}^n P_j(p^1,q^1)^2 \nonumber \\
&=\sum_{j=m_1+1}^n \left( \sum_{1\le l<i\le m_1} b^j_{l,i} (x_l y_i-x_i y_l) \right)^2 \nonumber \\
&= \sum_{j=m_1+1}^n \left( \sum_{1\le l<i\le m_1} b^j_{l,i} \langle (x_l,y_l), (y_i,-x_i) \rangle \right)^2\\
&\le m_2 m_1 (m_1-1) \max_{m_1+1\le j \le n} \max_{1\le l<i \le m_1} (b^j_{l,i})^2 |p^1-q^1|^2 |p^1+q^1|^2 \nonumber \\ &\le C_R^2 |p^1-q^1|^2 \nonumber,
\end{align}
where after \eqref{P} we used Cauchy-Schwarz and
\begin{equation}\label{CR}
C_R=  2 R \sqrt{m_2 m_1 (m_1-1) \max_{m_1+1\le j \le n} \max_{1\le l<i\le m_1} (b^j_{l,i})^2 }.
\end{equation}
Instead of line \eqref{P} we can also do the following.
\begin{align}\label{P2}
|P(p^1,q^1)|^2&= \sum_{j=m_1+1}^n \left( \sum_{1\le l<i\le m_1} b^j_{l,i} \langle (x_l,-x_i), (y_i,y_l) \rangle \right)^2 \nonumber\\ 
& \le  m_2 m_1 (m_1-1) \max_{m_1+1\le j \le n} \max_{1\le l<i \le m_1} (b^j_{l,i})^2 |p^1|^2 |q^1|^2 \nonumber \\
& \le C_R^2 |q^1|^2.
\end{align}

\begin{ex}\label{ex3}
For any $m_1\le s \le n$, there exist $c_s,\delta_s>0$ and a compact set $F_s \subset \mathbb{G}$ such that $\mathcal{H}^s_E(F_s)>0$, $\mathcal{H}^{2s-m_1}_\mathbb{G}(F_s) < \infty$ and for $\mathcal{H}^s_E$-almost every $p \in F_s$,
\begin{equation}\label{claimex3}
\liminf_{r\rightarrow 0} \frac{\mathcal{H}^s_E(F_s \cap B_E(p,r)\setminus V(p)( \delta_s r))}{r^s}\ge c_s.
\end{equation}
\end{ex}

This construction comes from Thereom 5.4 in \cite{RV}, where it is done in Carnot groups of any step, and it modifies the self-similar construction used in Proposition 4.14 in \cite{BTW}. We use the same notation, which we recall here. 
Let $m_1\le s \le n$. Let $A_1=\{0,1\}^{m_1}$, $A_2=\{0,1,2,3\}^{m_2}$ and define
\begin{displaymath}
\mathcal{F}_1=\{F_{a_1}: a_1 \in A_1 \}
\end{displaymath}
and
\begin{displaymath}
\mathcal{F}_2=\{F_{a_1 a_2}: a_1 \in A_1, a_2 \in A_2 \},
\end{displaymath}
where
\begin{displaymath}
F_{a_1}(p)= [a_1,0] \cdot \delta_{1/2}([-a_1,0]\cdot p)
\end{displaymath}
and
\begin{displaymath}
F_{a_1 a_2}(p)= [a_1, a_2] \cdot \delta_{1/2}([-a_1,-a_2 ]\cdot p).
\end{displaymath}
Then define a sequence $(n_i)_{i \in \mathbb{N}}$ by induction as follows. Let $n_1=2$ and, assuming that $n_1, \dots, n_t$ have been defined, let $n_{t+1}=2$ if
\begin{displaymath}
\prod_{i=1}^t n_i^{2 m_2} < 2^{2t(s -m_1)}
\end{displaymath}
and $n_{t+1}=1$ otherwise. If $n_i=1$ we will use the system $\mathcal{F}_1$ at step $i$, otherwise we will use $\mathcal{F}_2$. Write $\mathcal{F}_2=\{g_1, g_2, \dots, g_{2^{m_1+2m_2}} \}$, where $g_i \notin \mathcal{F}_1$ for $2^{m_1}<i\le 2^{m_1+2m_2} $. Let $I=\{ 1,2, \dots, 2^{m_1+2m_2}\}$, $I^n=\{\texttt{i}=(i_1, \dots, i_n):i_j \in I\}$, $I^*=\cup_{n=1}^\infty I^n$, $I^\infty=I^\mathbb{N}$ and
\begin{displaymath}
J=\{ \texttt{i}=(i_1, i_2, \dots): i_j \in \mathbb{N}, 1\le i_j \le n_j^{2m_2} 2^{m_1}\} \subset I^\infty.
\end{displaymath}
Then we write $J_n=\{\texttt{i} \in I^n: [\texttt{i}] \cap J \neq \emptyset \}$ and $J_*=\cup_{n=1}^\infty J_n$. Here $[\texttt{i}]$ is the cylinder set of $\texttt{i}$, that is the set of $\texttt{i}\texttt{j} \in I^\infty$ obtained by juxtaposing $\texttt{i} \in I^n$ and $\texttt{j}\in I^* \cup I^\infty$. For $\texttt{i}=(i_1, \dots, i_t) \in I^t$ let
\begin{displaymath}
X_\texttt{i}=g_{i_1 } \circ \dots \circ g_{i_t}(E),
\end{displaymath}
where $E$ is the attractor of $\mathcal{F}_2$. For $\texttt{i}=(i_1,\dots,i_n)\in I^n$ we denote by $|\texttt{i}|=n$ its length, by $\texttt{i}|_m \in I^m$ with $1\le m <n$ the element $(i_1, \dots, i_m)$ and by $\texttt{i}^-$ the element $\texttt{i}|_{n-1}$.
The collection $\{X_\texttt{i}: \texttt{i} \in J_*\}$ is a $(2 s - m_1)$-controlled Moran sub-construction, which means that it satisfies the following conditions:
\begin{enumerate}
\item[(i)] $X_\texttt{i}  \subset X_{\texttt{i}^-}$,
\item[(ii)] there exist a constant $D\ge 1$ and $n \in \mathbb{N}$ such that \begin{displaymath}
\max_{\texttt{i} \in J_n} \mbox{diam}_\infty (X_\texttt{i}) < 1/D,
\end{displaymath}
\item[(iii)] there exists a constant $C>0$ such that for every $\texttt{i} \in J_*$ and $m \in \mathbb{N}$,
\begin{equation}\label{Msc}
\frac{\mbox{diam}_\infty(X_\texttt{i})^{2 s-m_1}}{C} \le \sum_{\texttt{j}\in I^m, \texttt{ij} \in J_*} \mbox{diam}_\infty(X_{\texttt{ij}})^{2 s-m_1} \le C \mbox{diam}_\infty(X_\texttt{i})^{2 s-m_1}.
\end{equation}
\end{enumerate}
We refer to Section 5 in \cite{RV} for more details and examples of Moran sub-constructions. Since the maps $g_j$ are all contractions of ratio $1/2$ with respect to $d_\infty$, we have 
\begin{displaymath}
\mbox{diam}_\infty (X_\texttt{i}) \le \frac{1}{2^{|\texttt{i} |}}\mbox{diam}_\infty(E).
\end{displaymath}
Moreover, in \eqref{Msc} we can take $C=2^{4 m_2}$. Let $F_s$ be the limit set of this sub-construction. It is shown in Theorem 5.4 in \cite{RV} that $F_s$ is compact,
\begin{displaymath}
\mathcal{H}^{2s - m_1}_\mathbb{G}(F_s)< \infty \quad \mbox{and} \quad \mathcal{H}^s_E(F_s) >0.
\end{displaymath}

Let $\pi_1: \mathbb{G}\rightarrow \mathbb{R}^{m_1}$ and $\pi_2: \mathbb{G}\rightarrow \mathbb{R}^{m_2}$ be the orthogonal projections to the first and second layer respectively. Note that for every $a_1\in A_1$, $a_2 \in A_2$ and $p=[p^1,p^2] \in \mathbb{G}$,
\begin{displaymath}
\pi_1(F_{a_1}([p^1,p^2])=\frac{1}{2}(p^1-a_1) + a_1= \pi_1(F_{a_1 a_2}([p^1,p^2]),
\end{displaymath}
which means that $\pi_1(F_s)$ is the invariant set for the Euclidean self-similar iterated function system $\{ \pi_1(F_{a_1}): a_1 \in A_1\}$. The construction used in Proposition 4.14 in \cite{BTW} gives the same. It follows that $\pi_1(F_s)$ is a Euclidean self-similar set. It also satisfies the open set condition with the open set $(0,1)^{m_1}$, thus $\mathcal{H}^{m_1}_E(\pi_1(F_s))>0$ and it is $m_1$-Ahlfors regular.

Now for almost every $q^1 \in \pi_1(F_s) $ the set $\pi_2(F_s \cap \pi_1^{-1}(q^1))$ is a Euclidean translate of the limit set $K'$ of the Euclidean construction $\{ Y_\texttt{i}: \texttt{i} \in J'_*\}$ with
\begin{align*}
& Y_\texttt{i}=h_{i_1} \circ \dots \circ h_{i_m}([0,2]^{m_2}),\\
& J'=\{ \texttt{i}=(i_1,i_2, \dots ): 1 \le i_j \le n_j^{2 m_2} \}, \\
& h_{i_j}(y)= \frac{1}{4}y+\frac{3}{4}a_{i_j}.
\end{align*}
This is different from the construction in \cite{BTW}, where these level sets $\pi_2(F_s \cap \pi_1^{-1}(q^1))$ are Euclidean translates of an invariant set of a Euclidean self-similar iterated function system.
The collection $\{ Y_\texttt{i}: \texttt{i} \in J'_*\}$ is an $(s-m_1)$-Moran sub-construction, so it satisfies properties (i), (ii), (iii) above with respect to the Euclidean metric. In particular, there is a constant $C>0$ such that for every $\texttt{i} \in J'_*$ and $n \in \mathbb{N}$,
\begin{equation}\label{CMSC}
\frac{\mbox{diam}_E(Y_\texttt{i})^{ s-m_1}}{C} \le \sum_{\texttt{j} \in I^n, \texttt{ij} \in J'_*} \mbox{diam}_E(Y_{\texttt{ij}})^{ s-m_1} \le C \mbox{diam}_E(Y_\texttt{i})^{ s-m_1}.
\end{equation}
We can take again $C=2^{4m_2}$. In the proof of Theorem 5.4 in \cite{RV} it is also shown that $\mathcal{H}^{s-m_1}_E(K')>0$ since it satisfies the finite clustering property, which is a separation condition defined in Section 3 in \cite{RV} stating the following. Letting for $y \in K'$ and $r>0$,
\begin{displaymath}
Z(y,r)= \{ \texttt{i} \in J'_*: \mbox{diam}_E(Y_\texttt{i}) \le r < \mbox{diam}_E(Y_{\texttt{i}^-}), Y_\texttt{i} \cap B_E(y,r) \neq \emptyset\},
\end{displaymath}
we have
\begin{equation*}
Z=\sup_{y \in K'} \limsup_{r \rightarrow 0} \#Z(y,r) < \infty.
\end{equation*}
Writing $K' = \cup_{l=1}^\infty K'_l$, where $K'_l=\{ y \in K': \#Z(y,r) \le Z \ \forall \ 0<r<1/l\}$, we could run the argument below for every $l$ such that $\mathcal{H}^{s-m_1}_E(K'_l) >0$ so we can assume that $\#Z(y,r) \le Z$ for every $y \in K'$ and $0<r<r_0$ for a certain $r_0 >0$.

We claim that for every $y \in K'$ and every $0<r < r_0$ we have
\begin{equation}\label{ADrK'}
\frac{r^{s-m_1}}{C_s} \le \mathcal{H}^{s-m_1}_E(K'\cap B_E(y,r)) \le C_s r^{s-m_1},
\end{equation}
for some constant $C_s>0$. To this end, for $y \in K'$ and $0<r < r_0$, note that for every $m \in \mathbb{N}$
\begin{displaymath}
K' \cap B_E(y,r) \subset \bigcup_{\texttt{i}\in Z(y,r)} \bigcup_{\texttt{j} \in I^m, \texttt{i}\texttt{j} \in J'_*} Y_{ \texttt{i}\texttt{j}},
\end{displaymath}
hence by \eqref{CMSC}
\begin{align*}
\mathcal{H}^{s-m_1}_E(K' \cap B_E(y,r))& \le \sum_{\texttt{i}\in Z(y,r)} \sum_{\texttt{j} \in I^m, \texttt{i}\texttt{j} \in J'_*} \mbox{diam}_E( Y_{ \texttt{i}\texttt{j}})^{s-m_1}\\
& \le  \sum_{\texttt{i}\in Z(y,r)} C \mbox{diam}_E(Y_\texttt{i})^{s-m_1} \le ZC r^{s-m_1},
\end{align*}
which proves the right-hand side inequality in \eqref{ADrK'}. Let $ \texttt{i} \in J'$ be such that $y=\pi( \texttt{i})$, where $\pi$ is the projection mapping $J' \rightarrow K' $, $\{\pi(\texttt{i}) \}=\cap_{n=1}^\infty Y_{\texttt{i}|_n} $.
Note that \begin{displaymath}
\mbox{diam}_E(Y_{ \texttt{i}|_n}) =  \frac{2\sqrt{m_2}}{4^n}.
\end{displaymath}
Given $0<r<r_0$, let $n$ be such that 
\begin{displaymath}
\frac{2 \sqrt{m_2}}{4^{n}} \le r \le \frac{2 \sqrt{m_2}}{4^{n-1}},
\end{displaymath}
so that $Y_{ \texttt{i}|_n} \subset B_E(y,r)$. Let $\epsilon = \sqrt{m_2}/(C 4^{n(s-m_1)})$. There exists $m>n$ sufficiently large so that
\begin{align*}
\mathcal{H}_E^{s-m_1}(K' \cap B_E(y,r)) \ge \sum_{\texttt{j} \in I^{m-n}, \texttt{i}|_n \texttt{j} \in J'_m} \mbox{diam}_E(Y_{\texttt{i}|_n \texttt{j}})^{s-m_1} - \epsilon,
\end{align*}
noting that the sets $Y_{\texttt{i}|_n \texttt{j}}$ are disjoint.
Using \eqref{CMSC} it follows that
\begin{align*}
\mathcal{H}_E^{s-m_1}(K' \cap B_E(y,r)) \ge \frac{\mbox{diam}_E(Y_\texttt{i}|_n)^{ s-m_1}}{C}-\epsilon = \frac{2\sqrt{m_2}}{C4^{n(s-m_1)}}-\frac{2\sqrt{m_2}}{2C4^{n(s-m_1)}} \ge \frac{r^{s-m_1}}{C_s},
\end{align*}
which completes the proof of \eqref{ADrK'}.

Let $p =[p^1,p^2]\in F_s$ and 
\begin{equation}\label{rbounds}
0<r<\min\left\{r_0, \frac{c^2 C_R}{40(2+5C_R^2)}, \frac{C_R c^2 c_0}{2}\right\},
\end{equation}
where $C_R$ is the constant in \eqref{CR} and $c_0$ is defined below. Let $X_\texttt{i}$ be such that $p \in X_\texttt{i}$ and $\diam_\infty(E)/2^{|i|} \le r < \diam_\infty(E)/2^{|i|-1}$. Then $X_\texttt{i} \subset B_\infty(p,r)$.
For almost every $q =[q^1,q^2] \in X_\texttt{i}$ such that $|q^1-p^1| \le r/40(2+5C_R^2)$, we have the following. By redefining $K'$ as a translation of itself if needed, we can assume that $\pi_2(\pi_1^{-1}(q^1) \cap F_s)= K'+ q^2 \subset \mathbb{R}^{m_2}$. Then we let
\begin{equation*}
L_q=\{ [q^1, q^2+t] \in F_s: t \in K', c_0 r \le |t| \le  r/(2+5C_R^2) \}
\end{equation*}
for some constant $c_0$ such that 
\begin{equation}\label{HpiLq}
 \mathcal{H}^{s-m_1}_E(\pi_2(L_q)) \ge C'_s r^{s-m_1}.
\end{equation}
The existence of $c_0$ follows from the fact that $\pi_2(L_q) = (K'+q^2) \cap B_E(q^2,r/(2+5C_R^2))\setminus B_E(q^2,c_0r)$ and from \eqref{ADrK'}. Let
\begin{equation*}
D_r(p)=F_s \cap B_E(p,r) \cap \left\{ [u^1,u^2]: |u^1-p^1| \le \frac{r}{40(2+5C_R^2)} \right\} \setminus V(p)\left(\frac{c_0 r}{2\sqrt{2}C_R} \right).
\end{equation*}
We claim that 
\begin{equation}\label{LqsubsDr}
L_q \subset D_r(p) \quad \mbox{for almost every} \ q \in  B_\infty(p,r) \cap F_s, |q^1-p^1| \le \frac{r}{40(2+5C_R^2)}.
\end{equation}
Let $q_t=[q^1,q^2+t] \in L_q$. By \eqref{P}, the fact that $q \in X_\texttt{i} \subset B_\infty(p,r)$ and \eqref{rbounds} we have
\begin{align*}
|q^2-p^2| &\le |q^2-p^2-P(p^1,q^1)| + |P(p^1,q^1)| \\
&  \le \frac{d_\infty(p,q)^2}{c^2} + C_R |p^1-q^1|\le \frac{r^2}{c^2} + C_R \frac{r}{40(2+5C_R^2)} \le \frac{ C_R}{20(2+5C_R^2)} r.
\end{align*}
Hence
\begin{align*}
d_E(p, q_t)^2= |p^1-q^1|^2+|p^2-q^2-t|^2 \le \frac{r^2}{40^2(2+5C_R^2)^2} + \left( \frac{C_R}{20(2+5C_R^2)}r+ \frac{r}{2+5C_R^2}\right)^2 \le r^2,
\end{align*} 
which implies that $q_t \in B_E(p,r)$. Since $q_t \in F_s$, it remains to show that $d_E(q_t,V(p)) > c_0r/(2\sqrt{2}C_R)$. As in \eqref{distVp}, we have that
\begin{align*}
d_E(q_t,V(p))=\min_{s^1 \in \R^{m_1}} (|q^1-s^1|^2+|q^2+t-p^2-P(p^1,s^1)|^2)^{1/2}.
\end{align*} 
For every $s^1 \in \R^{m_1}$
\begin{align*}
|t|& \le |q^2+t-p^2-P(p^1,s^1)|+|-q^2+p^2+P(p^1,s^1)|\\
& \le|q^2+t-p^2-P(p^1,s^1)|+|-q^2+p^2+P(p^1,q^1)|+|P(p^1,s^1)-P(p^1,q^1)|\\
& \le |q^2+t-p^2-P(p^1,s^1)|+\frac{d_\infty(p,q)^2}{c^2}+C_R|s^1-q^1|,
\end{align*}
where in the last inequality we used \eqref{P2} as follows:
\begin{displaymath}
|P(p^1,s^1)-P(p^1,q^1)|=|P(p^1,s^1-q^1)| \le C_R |s^1-q^1|.
\end{displaymath}
Hence
\begin{align*}
d_E(q_t,V(p))&\ge \min_{s^1}\frac{1}{\sqrt{2}C_R}(C_R |s^1-q^1| + |q^2+t-p^2+P(p^1,s^1)|)\\
& \ge\frac{1}{\sqrt{2}C_R} \left( |t|-\frac{d_\infty(p,q)^2}{c^2} \right)\ge \frac{1}{\sqrt{2}C_R} \left(c_0r- \frac{r^2}{c^2}\right) >\frac{c_0}{2 \sqrt{2} C_R}r,
\end{align*}
where we used \eqref{rbounds}. This concludes the proof of \eqref{LqsubsDr}.
It follows that for almost every $q \in X_\texttt{i}$ with $|q^1-p^1| \le r/40(2+5C_R^2)$ there exists a point $[q^1, q^2+t] \in D_r(p)$, so $\pi_1(q)=\pi_1[q^1,q^2+t]) \in \pi_1(D_r(p))$. Hence 
\begin{align*}
\mathcal{H}^{m_1}_E(\pi_1(D_r(p))) &\ge \mathcal{H}^{m_1}_E(\pi_1(X_\texttt{i}\cap  \{ u^1 \in \mathbb{R}^{m_1}: |u^1-p^1| \le r/40(2+5C_R^2)\}))\\
& =\mathcal{H}^{m_1}_E( \pi_1( X_\texttt{i}) \cap \{ u^1 \in \mathbb{R}^{m_1}: |u^1-p^1| \le r/40(2+5C_R^2)\}).
\end{align*}
Using the fact that $\pi_1(F_s)$ is $m_1$-Ahlfors regular, we have
\begin{equation*}
 \mathcal{H}^{m_1}_E(\pi_1(D_r(p))) \ge C r^{m_1}
\end{equation*}
for some constant $C >0$. 
Hence by \eqref{HpiLq} and Theorem 7.7 in \cite{M} we have for some constant $c'>0$,
\begin{align*}
\mathcal{H}^s_E \left(F_s \cap B_E(p,r) \setminus V(p) \left( \frac{c_0r}{2\sqrt{2}C_R}\right) \right) &\ge \mathcal{H}^s_E(D_r(p)) \\
& \ge c' \int_{\pi_1(D_r(p))} \mathcal{H}^{s-m_1}_E(\{[0,q^2]: [q^1,q^2] \in D_r(p) \}) d \mathcal{H}^{m_1}_E(q^1) \\
& \ge c' \int_{\pi_1(D_r(p))} \mathcal{H}^{s-m_1}_E(\pi_2(L_q)) d \mathcal{H}^{m_1}_E(q^1) \\
& \ge c' C'_s r^{s-m_1} \mathcal{H}^{m_1}_E(\pi_1(D_r(p))) \ge c_s r^{s-m_1} r^{m_1} = c_s r^s,
\end{align*}
which proves \eqref{claimex3}.

\section{Some Remarks about general Carnot groups}\label{stepk}

We recall here few basic facts about general Carnot groups and show that the analogue of Theorem \ref{ssmall1} always holds. Let $\mathbb{G}$ be a Carnot group of step $k \ge 2$, which means that its Lie algebra $\mathfrak{g}$ admits a stratification of the form
\begin{displaymath}
\mathfrak{g}=V_1\oplus \dots \oplus V_k, \quad [V_1,V_i]=V_{i+1}, \quad V_k \neq \{0\}, \quad V_i=\{0\}, i >k.
\end{displaymath}
Using exponential coordinates, $\mathbb{G}$ can be identified with $\R^n = \R^{m_1} \times \dots \times \R^{m_k}$, where $m_i$ is the dimension of $ V_i$. We denote points by $p=[p^1, \dots, p^k]$ with $p^j \in \R^{m_j}$. The group operation has the form
\begin{align*}
p \cdot q=[p^1+q^1, p^2 + q^2+P_2(p,q), \dots, p^k+q^k+P_k(p, q)],
\end{align*}
where $P_j=(P_{j,1}, \dots, P_{j,m_j})$ and each $P_{j,i}$ is a polynomial of degree $j$ with respect to the dilations
\begin{displaymath}
\delta_\lambda(p)=[\lambda p^1, \lambda^2 p^2, \dots, \lambda^k p^k], \ \lambda>0.
\end{displaymath}
Moreover, for every $p,q \in \mathbb{G}$,
\begin{displaymath}
P_{j,i}(p,0)=P_{j,i}(0,q)=0, \quad P_{j,i}(p,p)=P_{j,i}(p,-p)=0
\end{displaymath}
and
\begin{displaymath}
P_{j,i}(p,q)=P_{j,i}([p^1, \dots, p^{j-1}],[q^1, \dots, q^{j-1}]).
\end{displaymath}

As in the case of step $2$, we consider the following metric, which is bi-Lipschitz equivalent to the Carnot-Carath\'eodory one:
\begin{align*}
d_\infty(p,q)=\max \{ |p^1-q^1|, \epsilon_2 |p^2-q^2+P_2(p,q)|^{1/2}, \dots, \epsilon_k |p^k-q^k+P_k(p,q)|^{1/k}\},
\end{align*}
where $\epsilon_j \in (0,1)$, $j=2, \dots k$, are constants depending only on the group structure (see Theorem 5.1 in \cite{FSSC}). By Proposition 5.15.1 in \cite{BLU} for every $0<R<\infty$ there exists $c_R>0$ such that for every $p,q \in B_E(0,R)$ we have
\begin{equation}\label{compds}
\frac{1}{c_R} d_E(p,q) \le d_\infty(p,q) \le c_R d_E(p,q)^{1/k}.
\end{equation}

For a Carnot group of step $k$ the functions $\beta_\pm$ appearing in the dimension comparison theorem (see Theorem 2.4 in \cite{BTW}) have the form
\begin{equation}\label{betaM}
\beta_-(s)=\begin{cases}
s, & \quad 0<s \le m_1\\
2s-m_1, & \quad m_1 < s \le m_1+m_2\\
\dots\\
ks -(k-1)m_1-\dots - m_{k-1},& \quad m_1 +\dots m_{k-1} < s \le n
\end{cases}
\end{equation}
and
\begin{equation}\label{betaP}
\beta_+(s)=\begin{cases}
k s, & \quad 0<s \le m_k\\
(k-1)s + m_k, & \quad m_k< s \le m_{k-1}+m_k\\
\dots\\
s + m_2+2m_3+ \dots+(k-1)m_s, & \quad m_2+\dots m_k <s \le n.
\end{cases}
\end{equation}
The local Hausdorff measure comparison \eqref{Hm} still holds with these functions $\beta_\pm$.
The horizontal $m_1$-dimensional subspace $V(p)$ passing through a point $p$ is given by those points $q=[q^1, \dots, q^k]  \in \mathbb{G}$ such that
\begin{align}\label{Vps}
q^j= p^j+P_j(p,q), \qquad j=2, \dots, k.
\end{align} 
Note that if $q \in V(p)$, then $d_\infty(p,q)=|p^1-q^1| \le d_E(p,q)$.  

As you can see form \eqref{betaM} and \eqref{betaP}, for general Carnot groups there would be many cases to consider for the dimension comparison (double the step of the group) and we do not know what kind of geometric properties extremal sets should satisfy in each case. It seems natural to believe that also here when $s \le m_1$ sets with positive and finite Euclidean and homogeneous $s$-Hausdorff measures must be in some sense horizontal but we do not know how to prove it. However, we can show that in the case when $0<s<m_k$, which implies $\beta_+(s)=ks$, sets with finite $s$-dimensional Euclidean Hausdorff measure must be in some sense vertical.
Indeed, the following version of Theorem \ref{ssmall1} holds.

\begin{theorem}
Let $0 < s<m_k$ and $A \subset \mathbb{G}$ be such that $\mathcal{H}^s_E(A)<\infty$. Then  for $\mathcal{H}^{k s}_\mathbb{G}$ almost every $p \in A$ there exists $\delta>0$ such that
\begin{equation}
\limsup_{r \rightarrow 0} \frac{\mathcal{H}^s_E (A \cap B_E(p,r) \setminus V(p)(\delta r))}{(2r)^s} > 0.
\end{equation}
\end{theorem}

The proof is basically the same as in the case of step $2$ groups with few small changes that we explain here. The first change is in equation \eqref{change}, which becomes using \eqref{compds},
\begin{equation*}
\mbox{diam}_\infty(B_E(p_i,r_i))\le c_R \mbox{diam}_E(B_E(p_i,r_i))^{1/k}\le c_R (2\eta)^{1/k}.
\end{equation*}
Then, taking $\eta'= c_R (2\eta)^{1/k}$, \eqref{diamH} becomes
\begin{align}
 \mathcal{H}^{ks}_{\mathbb{G},\eta'}(A'\cap  B_E(p_i,r_i) \cap V(p_i)(\delta r_i)) &\le \mbox{diam}_\infty(B_E(p_i,r_i) \cap V(p_i)(\delta r_i))^{ks} \nonumber \\ & \le (2( c_R+2) (\delta r_i)^{1/k})^{ks}= C'' (\delta r_i)^s. 
\end{align}
This can be proved in the same way as \eqref{diamH} where \eqref{change2} becomes
\begin{equation*}
d_\infty(q,\bar{q}) \le c_Rd_E(q,\bar{q})^{1/k} \le c_R(\delta r_i)^{1/k}, \quad d_\infty(q',\bar{q}') \le c_R d_E(q',\bar{q}')^{1/k} \le c_R(\delta r_i)^{1/k}
\end{equation*} 
and we use the fact that $r_i < \delta$ implies $r_i \le (\delta r_i)^{1/k}$. The rest of the proof is exactly the same as for step $2$ groups.

\vspace{1cm}
\begin{footnotesize}
{\sc Department of Mathematics and Statistics,
P.O. Box 68,  FI-00014 University of Helsinki, Finland,}\\
\emph{E-mail addresses:} 
\verb"laura.venieri@helsinki.fi" 

\end{footnotesize}

\end{document}